\documentclass[a4paper,11pt]{amsart}
\usepackage[english]{babel}
\usepackage{amsmath,amssymb,amsthm,mathrsfs,amsopn,amscd}

\usepackage{enumerate}
\usepackage{verbatim,graphicx,color,xcolor}
\usepackage[utf8]{inputenc}
\usepackage{indentfirst}
\usepackage[T1]{fontenc}
\usepackage{fancyhdr}

\usepackage{hyperref}

\newtheorem{thm}{Theorem}[section]

\newtheorem{remark}[thm]{Remark}

\def\RN{\mathbb{R}^N}
\def\RR{\mathbb{R}}

\def\al{\alpha}
\def\la{\lambda}

\def\cQ{\mathcal{Q}}

\def\pa{\partial}
\def\fhi{\varphi}
\def\de{\delta}

\newcommand{\De}{\Delta}

\newcommand{\dve}{\mathrm{div}}

\newcommand{\cG}{\mathcal G}

\newcommand{\ga}{\gamma}

\newcommand{\si}{\sigma}
\newcommand{\Si}{\Sigma}
\newcommand{\Om}{\Omega}

\newcommand{\RE}{\mathbb R}

\newcommand{\ovr}{\overline}

\title[]{The Matzoh Ball Soup problem: \\ 
a complete characterization}
\thanks{This research has
been supported by the Gruppo Nazionale per l'Analisi Matematica, la Probabilit\`a e le loro Applicazioni (GNAMPA) of the italian Istituto Nazionale di Alta Matematica (INdAM).} 

\author[Magnanini]{Rolando Magnanini}
\address{Dipartimento di Matematica e Informatica ``U. Dini'', Universit\`a di Firenze, viale Morgagni 67/A, 50134 Firenze, Italy}
\email{magnanin@math.unifi.it}

\author[Marini]{Michele Marini}
\address{Scuola Normale Superiore, Piazza dei Cavalieri 7, 56126 Pisa, Italy}
\email{michele.marini@sns.it}


\begin{document}

\maketitle

\begin{abstract}
We characterize all the solutions of the heat equation that have their (spatial) equipotential surfaces
which do not vary with the time. Such solutions are either isoparametric or split in space-time. The result gives a final answer to a problem raised by M. S. Klamkin, extended by G. Alessandrini, and
that was named the {\it Matzoh Ball Soup Problem} by L. Zalcman. Similar results can also be drawn for a
class of quasi-linear parabolic partial differential equations with coefficients which are homogeneous functions of the gradient variable. This class contains the (isotropic or anisotropic) evolution $p$-Laplace and normalized $p$-Laplace equations.
\end{abstract}

\pagestyle{plain}
\thispagestyle{plain}
\markboth{}{}

\section{Introduction}
\label{introduction}
In this paper we settle a question raised by M.S. Klamkin in \cite{Kl} and extended in \cite{Al2}.

\begin{quote}{} {\bf Klamkin's conjecture (1964).} 
Consider the heat conduction problem for a solid $\Om$, 
$$
u_t=\De u \ \mbox{ in } \ \Om\times(0,\infty).
$$
Initially,
$u=0.$ On the boundary $u=1.$ The solution to the problem is well-known for a sphere
and, as to be expected, it is radially symmetric. Consequently, the equipotential
surfaces do not vary with the time (the temperature on them, of course, varies). It
is conjectured for the boundary value problem above, that the sphere is the only
bounded solid having the property of invariant equipotential surfaces.
If we allow unbounded solids, then another solution is the infinite
right circular cylinder which corresponds to the spherical solution in 
two-dimensions. 
\end{quote}

L. Zalcman \cite{Za} included this problem in a list of questions about the ball and named it the {\it Matzoh Ball Soup} problem. For the case of a bounded solid, the conjecture was given a positive answer by G. Alessandrini \cite{Al1}: {\it the ball is the only
bounded solid having the property of invariant equipotential surfaces}. 
\par
In \cite{MS1}, the second author of this paper and S. Sakaguchi 
showed that to obtain the spherical symmetry of the solid in Klamkin's setting, it is enough to require that the solution has {\it only one} invariant equipotential  surface (provided this surface is a $C^1$-regular boundary of  domain). In a subsequent series of papers, the same authors extended their result in several directions: spherical symmetry also holds for certain evolution nonlinear equations (\cite{MS2, MS4, MS6, MS7}); a hyperplane can be characterized as an invariant equipotential surface in the case of an unbounded solid that satisfies suitable sufficient conditions (\cite{MS3, MS5}); for a certain Cauchy problem, a helicoid is a possible invariant equipotential surface  (\cite{MPS});
spheres, infinite cylinders and planes are characterized as (single) invariant equipotential surfaces in $\RE^3$ (\cite{MS8}); similar symmetry results can also be proven in the sphere and the hyperbolic space (\cite{MS4}).
\par
In \cite{Al2}, G. Alessandrini re-considered Klamkin's problem for a bounded domain in the case in which $u$ initially equals any function $u_0\in L^2(\Om)$ and is zero on $\pa\Om$ for all times. He discovered that either $u_0$ is a Dirichlet eigenfunction or $\Om$ is a ball. A comparable result was obtained by S. Sakaguchi \cite{Sa} when a homogeneous Neumann condition is in force on $\pa\Om$.
\par
The aim of this paper is to show that Klamkin's property of having invariant equipotential surfaces 
characterizes a solution of the heat equation without assuming {\it any whatsoever} initial or boundary condition. This is the content of our main result.
\begin{thm}
\label{th:stationary}
Let $\Om\subseteq\RN$ be a domain and let $u$ be a solution of the heat equation:
\begin{equation}
\label{heat}
u_t=\De u \ \mbox{ in } \ \Om\times(0,\infty).
\end{equation}
\par
Assume that there exists a $\tau>0$ such that, for every $t>\tau,$
$u(\cdot,t)$ is constant on the level surfaces of $u(\cdot,\tau)$ and $Du(\cdot,\tau)\ne 0$ in $\Om$.
\par
Then one of the following occurrences holds:
\begin{enumerate}[(i)]
\item
the function $\fhi=u(\cdot,\tau)$ (and hence $u$) is isoparametric, that is there exist two real-valued 
functions $f$ and $g$ such that $\fhi$ is a solution of the following system of equations:
$$
|D\fhi|^2=f(\fhi) \ \mbox{ and } \ \De \fhi=g(\fhi) \ \mbox{ in } \ \Om;
$$
\item
there exist two real numbers  $\la, \mu$  such that
$$
u(x,t)=e^{-\la t} \phi_\la(x)+\mu, \ \ (x,t)\in\Om\times[\tau,\infty),
$$ 
where
$$
\De\phi_\la+\la\,\phi_\la=0 \ \mbox{ in } \ \Om;
$$
\item
there exists a real number $\ga$ such that 
$$
u(x,t)=\ga\,t+w(x), \ \ (x,t)\in\Om\times[\tau,\infty),
$$
where
$$
\De w=\ga \ \mbox{ in } \ \Om.
$$
\end{enumerate}
\end{thm}
As we shall see in Remark \ref{rmk:null gradient}, the assumption on the gradient of the solution $u$ is only technical. Without that assumption,  another (trivial) occurrence would be that $u$ is constant.
\par
The proof of Theorem \ref{th:stationary} relies on and completes those contained in \cite{Al2} and \cite{Sa}: there, option (iii) and the fact that initial and boundary conditions are unnecessary were overlooked. 
\par
{\it Isoparametric functions} are well-known in the literature; accordingly, their level surfaces are called
{\it isoparametric surfaces} and can also be characterized as those surfaces whose principal curvatures are all constant. The classical results of T. Levi-Civita \cite{Le} and B. Segre \cite{Se} classify  isoparametric functions in $\RN$ by their level surfaces:
they can be either concentric spheres, co-axial spherical cylinders (that is cartesian products of an $M$-dimensional sphere by an $(N-M-1)$-dimensional euclidean space, $0\le M\le N-2$), or parallel hyperplanes (affine spaces of co-dimension $1$). 
By using this fact, one can conclude fairly easily that, in Klamkin's setting, that is when (i) of Theorem \ref{th:stationary} holds and $u$ is constant on $\pa\Om$, then the possible shapes of a domain $\Om$ can be one of the following: a ball, its exterior or a spherical annulus; a spherical cylinder, its exterior or a cylindrical annulus; a half-space or an infinite strip (see \cite{Al2} for the simply connected case). Analogous results can  be drawn in the case we impose on $u$ a homogeneous Neumann condition (see \cite{Sa}),
\begin{equation}
\label{neumann}
u_\nu=0 \ \mbox{ on } \ \pa\Om\times(0,\infty).
\end{equation}
\par
Thus, a caloric function that has invariant equipotential surfaces enjoys some {\it splitting property in space-time}, since it is always separable (either with respect to addition or to multiplication).  We point out that 
this behaviour is not restricted to the case of the heat equation, since it also occurs for other {\it linear} evolution equations, such as the {\it wave equation}, the {\it Schr\"odinger equation} or 
any partial differential equation connected to the heat equation by some integral transform.
\par
More interestingly, a similar behaviour holds for the following class of quasi-linear evolution equations:
\begin{equation}
\label{non-linear}
u_t=\cQ\,u \ \mbox{ in } \ \Om\times(0,\infty),
\end{equation}
where the operator
\begin{equation}
\label{defQ}
\cQ\,u=\sum_{i,,j=1}^N a_{ij}(Du)\,u_{x_i x_j}
\end{equation}
is elliptic and the coefficients $a_{ij}(\xi)$ are sufficiently smooth $\al$-homogeneous functions of $\xi\in\RN$, $\al>-1$, that is such that
\begin{equation}
\label{homogeneous}
a_{ij}(\si\,\xi)=\si^\al a_{ij}(\xi) \ \mbox{ for every } \ \xi\in\RN,\ \si>0 \ \mbox{ and } \ i, j=1,\dots, N;
\end{equation}
this will be shown in Theorem \ref{th:non-linear}.
\par
Important instances of \eqref{non-linear} are the {\it evolution $p$-Laplace} equation,
\begin{equation}
\label{p-Laplace}
u_t=\dve\{ |Du|^{p-2}\,Du\} \ \mbox{ in } \ \Om\times(0,\infty);
\end{equation}
the  
{\it normalized evolution $p$-Laplace} equation,
\begin{equation}
\label{normalized-p-Laplace}
u_t=|Du|^{2-p}\dve\{ |Du|^{p-2}\,Du\} \ \mbox{ in } \ \Om\times(0,\infty);
\end{equation}
the (anisotropic) {\it evolution h-Laplace equation},
\begin{equation}
\label{H-Laplace}
u_t=\De_h u\ \mbox{ in } \ \Om\times(0,\infty);
\end{equation}
here, 
$$
\De_h u=\dve\{ h(Du)\,Dh(Du)\} 
$$
is the so-called {\it anisotropic} $h$-laplacian, where $h:\RN\to\RE$ is a non-negative convex function of class
$C^2(\RN\setminus\{0\})$, which is even and positively homogeneous of degree $1$.
\par
In Section \ref{sec:symmetry}, besides a brief discussion on the possible symmetries of the domain $\Om$ when initial and boundary conditions are added, we will also prove a classification theorem, related to equation \eqref{H-Laplace}, for isoparametric surfaces in the spirit of Levi-Civita and Segre's result (Theorem \ref{th:H-isoparametric}).

\setcounter{equation}{0}

\section{The splitting property}

The proof of Theorem \ref{th:stationary} and Theorem \ref{th:non-linear} below are similar; however,
we prefer to give them separately to draw the reader's attention on the different behaviours that
depend on the parameter $\al$ in \eqref{homogeneous}.

\begin{proof}[Proof of Theorem \ref{th:stationary}]
As originally observed in \cite{Al1}, the assumption on the level surfaces of $u$ implies that, if we set $\fhi(x)=u(x,\tau)$, then
there exist a number $T>\tau$ and a function $\eta:\RE\times[\tau,T)\to\RE$, with
\begin{equation}
\label{init-eta}
\eta(s,\tau)=s, \ \ s\in\RE,
\end{equation}
such that
\begin{equation}
\label{alessandrini}
u(x,t)=\eta(\fhi(x),t)
\end{equation} 
for every $(x,t)\in\ovr{\Om}\times[\tau,T);$  
thus, \eqref{heat} gives that
$$
\eta_{ss}(\fhi,t)\,|D\fhi|^2+\eta_s(\fhi,t)\,\De\fhi=\eta_t(\fhi,t)\ \mbox{ in } \Om\times[\tau,T).
$$
By differentiating in $t$ this identity, we obtain that $\fhi$ must satisfy in $\Om$ the
following system of equations:
\begin{equation}
\label{sakasys}
\begin{array}{ll}
&\eta_s(\fhi,t)\,\De\fhi+\eta_{ss}(\fhi,t)\,|D\fhi|^2=\eta_t(\fhi,t),\\
&\eta_{st}(\fhi,t)\,\De\fhi+\eta_{sst}(\fhi,t)\,|D\fhi|^2=\eta_{tt}(\fhi,t),
\end{array}
\end{equation}
for $t\in[\tau,T)$.
\par
Notice that the necessary smoothness of the function 
$\eta$ can be proved by a standard finite difference argument:
in fact, one can prove that, since $D\fhi\ne0$, then $\eta\in C^\infty(I\times[\tau,T))$, where $I=(\inf_\Om\fhi,\sup_\Om\fhi)$
and $\eta_s>0$ on $I\times[\tau,T)$ (see \cite[Lemma 1]{Al1}, \cite[Lemma 2.1]{Al2} or \cite[Lemma 2.1]{Sa} for details).
\par
As observed in \cite{Sa}, for the system \eqref{sakasys}, it is enough to consider the alternative cases in which the determinant
$$
D(s,t)=\eta_s\eta_{sst}-\eta_{st}\eta_{ss}
$$
is zero or not zero. In fact, if $D(s,t)\ne0$ at some $(s,t)\in I\times[\tau,T)$, then $D\ne0$ in an open neighborhood, say $U\times V\subset I\times [\tau,T)$, of $(s,t)$; \eqref{sakasys} then implies that 
$$
|D\fhi|^2=f(\fhi) \ \mbox{ and } \ \De \fhi=g(\fhi)  
$$
at least in a subdomain $\Om'$ of $\Om$, and the expressions  of $f$ and $g$ are given by the formulas
$$
f=\frac{\eta_s\eta_{tt}-\eta_{st}\eta_{t}}{\eta_s\eta_{sst}-\eta_{st}\eta_{ss}}, \quad
g=\frac{\eta_{t}\eta_{sst}-\eta_{tt}\eta_{ss}}{\eta_s\eta_{sst}-\eta_{st}\eta_{ss}};
$$
clearly, $f$ and $g$ are analytic functions. Thus $u(\cdot,t)$ (and hence $\fhi$) is isoparametric in an open subdomain of $\Om$; by using the classification result for isoparametric functions by Levi-Civita and Segre and the analiticity of $\fhi$, we can then conclude that the function $\fhi$ is isoparametric in the whole $\Om$.
\par
Otherwise, we have $D=0$ in $I\times [\tau, T)$. Thus, since
$$
\frac{\pa^2}{\pa s \pa t}\,\log (\eta_s)=(\eta_s)^{-2}\,(\eta_s\eta_{sst}-\eta_{st}\eta_{ss})=0
\ \mbox{ in } \ I\times[\tau,T),
$$
we have that $\log\eta_s(s,t)$ splits up into the sum of a function of $s$ plus a function of $t$;  \eqref{init-eta} then implies that $\eta_s(s,t)$ only depends on $t$, and hence it is easy to conclude that
$$
\eta(s,t)=a(t)\,s+b(t), \ \ (s,t)\in I\times[\tau,T),
$$
for some smooth functions $a$ and $b$ such that
$$
a(\tau)=1 \ \mbox{ and } \ b(\tau)=0.
$$
\par
Now, we now know that
$$
u(x,t)=a(t)\,\fhi(x)+b(t)
$$
is a solution of \eqref{heat}, thus \eqref{sakasys} can be written as
a linear system of equations: 
\begin{equation}
\label{sakasys2}
\begin{array}{lll}
&&a'(t)\,\fhi(x) -a(t)\,\De\fhi=-b'(t), \\
&&a''(t)\,\fhi(x) -a'(t)\,\De\fhi=-b''(t).
\end{array}
\end{equation}
The determinant of this system must be zero, otherwise $\fhi$
would be constant (in fact, it would be that $\fhi(x)$ is a function of $t$); thus,
$$
a(t) a''(t)-a'(t)^2=0 \ \mbox{ for } \ t\in[\tau,\infty), \ a(\tau)=1.
$$
All solutions of this problem can be written as $a(t)=e^{-\la (t-\tau)}$ for $\la\in\RE$ and,
by going back to \eqref{sakasys2}, we obtain that
$$
\De\fhi(x)+\la\,\fhi(x)=b'(t)\,e^{\la (t-\tau)}=\ga,
$$
for some constant $\ga$. Since $b(\tau)=0$, we have:
$$
b(t)=\ga\,\frac{1-e^{-\la (t-\tau)}}{\la} \ \mbox{ if } \ \la\not=0 \quad \mbox{ and } \quad
b(t)=\ga\, (t-\tau) \ \mbox{ if } \ \la=0.
$$
Therefore, we have obtained:
\begin{eqnarray*}
&&u(x,t)=e^{-\la (t-\tau)}\, [\fhi(x)-\ga/\la]+\ga/\la\ \mbox{ if } \ \la\not=0, \\
&&u(x,t)=\fhi(x)+\ga (t-\tau) \ \mbox{ if } \ \la=0.
\end{eqnarray*}
\par
In conclusion, by setting $\phi_\la=e^{\la\tau}[\fhi-\ga/\la]$ and $\mu=\ga/\la$ for $\la\not=0$, we get case (ii), while setting $w=\fhi-\ga \tau$  for $\la=0$ yields case (iii).
\end{proof}

\begin{remark}
\label{rmk:null gradient}
{\rm
For the sake of simplicity, in Theorem \ref{th:stationary} we assumed that $Du\ne 0$ in the whole $\Om$, in order to be able to deduce the necessary regularity for $\eta$.  Here, we will show how 
that assumption can be removed.
\par
The only possible obstruction to the regularity of $\eta$ is the presence of level sets of $\fhi$ whose points are all critical. Indeed $\eta$ is always smooth in the $t$ variable and is smooth in the $s$ 
variable, for every semi-regular value $s$ of $\fhi$ ({\em i.e.} a value that is the image of at least one regular point).
For every open subset of $\Om$ of regular points for $\fhi$ corresponding to a semiregular value of $\fhi$, the above theorem still holds true even if we remove the assumption on the gradient of the solution $u$. 
\par
We now show that, in any case, the classification holds true globally in $\Om$.
For simplicity we assume that there exists only one critical level set. 
\par
Let then $s$ be a value in the range of $\fhi$ such that $D \fhi(x)=0$, for every $x\in\fhi^{-1}(s)$ and set $\Om^+=\{x\in\Om\, :\, \fhi(x)> s\}$, and $\Om^-=\{x\in\Om\, :\, \fhi(x)< s\}$. Clearly the above classification holds true separately in $\Om^+$ and in $\Om^-$. Being cases (i), (ii) and (iii) closed relations involving continuous functions, if the same case occours both in $\Om^+$ and $\Om^-$, then it holds in the whole domain $\Om$. Thus, we are left with all the cases in which in $\Om^+$ and $\Om^-$ the solution assumes two different representations of those given in Theorem \ref{th:stationary}.
\par
We proceed by direct ispection.
If case (i) occurs in $\Om^+$ or in $\Om^-$, as it has been already remarked, then $\fhi$ extends to an isoparametric function in an open domain containing $\Om$.
\par
We then have only to study the case in which instances (ii) and (iii) are in force in $\Om^-$ and $\Om^+$, respectively; by contradiction, we shall see that it is not possible to have a critical level set. In fact, up to sum a constant, we can assume without loss of generality that $s\ne 0$. As shown in \cite[Lemma 2.2]{Al2}, the presence of a critical level set implies that $\De \fhi(x)=0$, for every $x\in\ovr{\Om^+}\cap\ovr{\Om^-}$. Then, according to the previous computations, we get that $0=\De \fhi(x)=\lambda \fhi(x)$, and, being $\lambda \ne 0$, we have that $\fhi(x)=0$ and thence $s=0$, that is a contradiction.
}
\end{remark}

We now consider the case of the quasi-linear operator \eqref{non-linear}. Notice that, when $\al=0$,
this also concerns \eqref{normalized-p-Laplace} besides the heat equation.
In what follows, it is useful to define the generalized gradient operator:
\begin{equation}
\label{grad}
\cG u=\sum_{i,j=1}^N a_{ij}(Du)\, u_{x_i} u_{x_j}.
 \end{equation}
 
\begin{thm}
\label{th:non-linear}
Let $\Om\subseteq\RN$ be a domain and let $u\in C^1((0,\infty); C^2(\Om))$ be a solution of equation \eqref{non-linear},
$$
u_t=\cQ\,u \ \mbox{ in } \ \Om\times(0,\infty),
$$
where the operator $\cQ$, given in \eqref{defQ}, is elliptic with coefficients $a_{ij}(\xi)$ that satisfy \eqref{homogeneous}.
\par
Assume that there exists a $\tau>0$ such that, for every $t>\tau,$
$u(\cdot,t)$ is constant on the level surfaces of $u(\cdot,\tau)$ and $Du(\cdot,\tau)\ne 0$ in $\Om.$
\par
Then the there exists a countable set of values $a_1<a_2<\ldots a_i<\ldots$, such that 
$$
\Om=\bigcup_{i\in\mathbb N}\fhi^{-1}([a_i,a_{i+1}])
$$ 
and, for every subdomain $\Om'\subset \fhi^{-1}([a_i,a_{i+1}])$ one of the following cases occurs:
\begin{enumerate}[(i)]
\item
there exist two real-valued 
functions $f$ and $g$ such that $\fhi=u(\cdot,\tau)$ is a solution of the following system of equations:
$$
\cG \fhi=f(\fhi) \ \mbox{ and } \ \cQ \fhi=g(\fhi) \ \mbox{ in } \ \Om';
$$
\item
there exist two real numbers  $\la, \mu$  such that
$$
u(x,t)=[1+\la\, (t-\tau)]^{-1/\al} \phi_\la(x)+\mu, \ \ (x,t)\in\Om'\times[\tau,\infty),
$$ 
with
$$
\cQ \phi_\la+\frac{\la}{\al}\,\phi_\la=0 \ \mbox{ in } \ \Om',
$$
if $\al\not=1$;
$$
u(x,t)=e^{-\la (t-\tau)} \phi_\la(x)+\mu, \ \ (x,t)\in\Om'\times[\tau,\infty),
$$ 
with
$$
\cQ \phi_\la +\la\,\phi_\la=0 \ \mbox{ in } \ \Om',
$$
if $\al=1$;
\item
there exists a real number $\ga$ such that 
$$
u(x,t)=\ga\,(t-\tau)+w(x), \ \ (x,t)\in\Om'\times[\tau,\infty),
$$
where
$$
\cQ w=\ga \ \mbox{ in } \ \Om'.
$$
\end{enumerate}
\end{thm}

\begin{proof}
The proof runs similarly to that of Theorem \ref{th:stationary}. We still begin by setting
$\fhi(x)=u(x,\tau)$ and $u(x,t)=\eta(\fhi(x),t)$, where $\eta$ satisfies \eqref{init-eta}. 
Since $Du\ne 0$ equation \eqref{non-linear} is uniformly parabolic, by standard parabolic regularity (see, for instance \cite{LU}), we have the necessary regularity to give sense to the following computations.
\par
By arguing as in the proof of Theorem \ref{th:stationary}, we obtain the system of equations:
\begin{equation}
\label{magnasys}
\begin{array}{ll}
&\xi(\fhi,\tau)\,\cQ \fhi+\xi_s(\fhi,\tau)\,\cG \fhi=\eta_t(\fhi,\tau),\\
&\xi_t(\fhi,\tau)\,\cQ \fhi +\xi_{st}(\fhi,\tau)\,\cG \fhi=\eta_{tt}(\fhi,\tau),
\end{array}
\end{equation}
where $\xi=(\eta_s)^{\al+1}$.
\par
At this point, the proof is slightly different from that of Theorem \ref{th:stationary}. If there exists $(s,t)$ such that $D(s,t)=\xi\,\xi_{st}-\xi_s\,\xi_t\not=0$, then $D(u(x,t),t)\not=0$ for $x\in\Om'=\{ x\in\Om: u(x,t)\in (s-\de, s+\de)\}$ for some $\de>0$. By setting
$$
\ovr f=\frac{\xi_s\,\eta_{tt}-\xi_{t}\,\eta_{t}}{\xi\,\xi_{st}-\xi_s\,\xi_t}, \quad
\ovr g=\frac{\xi_{st}\,\eta_{t}-\xi_{s}\,\eta_{tt}}{\xi\,\xi_{st}-\xi_s\,\xi_t},
$$
case (i) holds true for the function $u(\cdot,t)$. The same fact holds true for $\fhi$, since $u(x,t)=\eta(\fhi(x),t)$ and hence $\fhi$ and $u(x,t)$ share the same level sets; thus, there exist functions $f$ and $g$ such that $\cG \fhi=f(\fhi)$ and $\cQ \fhi=g(\fhi)$ in $\Om'$.
\par
Now, we define $J$ as the closure of the complement in the image of $u$ of the closure of the set $\{ s\,:\,D(s,t)\not=0,\,\,\mbox{ for some }t\}$; $J$ is a countable union of disjoint intervals. Let $[a,b]\subseteq J$; in $\Om'=\fhi^{-1}([a,b])$, we get
\begin{eqnarray*}
&&\frac{\pa^2}{\pa s \pa t}\,\log (\xi)=(\xi)^{-2}\,(\xi\,\xi_{st}-\xi_{t}\,\xi_{s})=0
\ \mbox{ in } \ [a,b]\times[\tau,T), \\ 
&&\xi(s,\tau)=1 \ \mbox{ for } \ s\in [a,b].
\end{eqnarray*}
As before, we obtain
$$
\eta(s,t)=a(t)\,s+b(t),
$$
with $a(\tau)=1$ and $b(\tau)=0$. Proceeding as before with $u(x,t)=a(t)\,\fhi(x)+b(t)$ gives the system:
\begin{equation}
\label{magnasys2}
\begin{array}{ll}
&a'(t)\,\fhi(x) -a(t)^{\al+1}\,\cQ \fhi=-b'(t),\\
&a''(t)\,\fhi(x) -(\al+1)\, a(t)^{\al} a'(t)\,\cQ \fhi=-b''(t).
\end{array}
\end{equation}
\par
The determinant of this system must be zero, otherwise $\fhi$
would be constant; thus, $a$ must satisfy the problem
$$
a^{\al+1} a''-(\al+1)\, a^{\al} (a')^2=0 \ \mbox{ in } \ [\tau,T), \ \ a(\tau)=1.
$$
The solutions of this problem are for $t\in[\tau, T)$
\begin{eqnarray*}
&&a(t)=[1+\la\, (t-\tau)]^{-1/\al} \ \mbox{ if } \ \al\not=0,\\
&&a(t)=e^{-\la\, (t-\tau)} \ \mbox{ if } \ \al=0,
\end{eqnarray*}
for some $\la\in\RR$, and going back to the first equation in \eqref{magnasys2} gives
\begin{eqnarray*}
&&\cQ \fhi+\frac{\la}{\al}\,\fhi(x) =b'(t)\,a(t)^{-(\al+1)}=\ga \ \mbox{ if } \ \al\not=0,\\
&&\cQ \fhi+\la\,\fhi(x) =b'(t)\,a(t)^{-1}=\ga \ \mbox{ if } \ \al=0.
\end{eqnarray*}
Thus, we have that
\begin{eqnarray*}
&&b(t)=\frac{\ga \al}{\la}\,\{1-[1+\la\, (t-\tau)]^{-1/\al}\}\ \mbox{ if } \ \al\not=0,\\
&&b(t)=\ga\,\frac{1-e^{-\la\, (t-\tau)}}{\la}\ \mbox{ if } \ \al=0,
\end{eqnarray*}
for $\la\not=0$,
$$
b(t)=\ga\,(t-\tau),
$$
for $\la=0$ and any $\al\ge 0$.
\par
Therefore, (ii) follows when $\la\not=0$, by setting $\phi_\la=\fhi-\ga \al/\la$ and $\mu=\ga \al/\la$ for $\al\not=0$ and $\phi_\la=\fhi-\ga/\la$ and $\mu=\ga/\al$ for $\al=0$; (iii) follows when $\la=0$, by choosing $w=\fhi$.
\end{proof}

\setcounter{equation}{0}

\section{Symmetry in space}
\label{sec:symmetry}

Theorems \ref{th:stationary} and \ref{th:non-linear} thus give us a catalog of functions having time-invariant equi-potential surfaces. In particular, they inform us that the shape of their level surfaces can be quite arbitrary, as those of eigenfunctions are.
\par
This arbitrariness significantly reduces if we add  special initial and boundary conditions. For instance, if we consider the situation in Klamikin's conjecture, that is if $\Om$ is bounded and simply connected and we require that the solution $u$ of
\eqref{heat} is initially $0$ on $\Om$ and equals a non-zero constant on $\pa\Om$ at all times\footnote{It can be seen that that constant can be replaced by any bounded positive function that bounded away from zero on $\pa\Om$.}, then the 
only possibility is that the level surfaces of $u$ be concentric spheres, as proved in \cite{Al1} and \cite{Sa}. If we examine the case of a {\it non-simply connected} bounded domain, then the connected components of its boundary are spheres and must be concentric, because they are parallel. Thus, a spherical annulus is the only possibility. In a similar way we argue when is $\Om$ is unbounded, 
thus obtaining the remaining possible solutions: a (solid) spherical cylinder; a hyperspace; an infinite strip; a cylindrical annulus; the exterior of a ball and of a spherical cylinder.
\par
Actually, the presence of those initial and boundary conditions allows one to relax the stringent 
requirement that {\it all} equi-potential surfaces be time-invariant. In fact, to obtain the spherical symmetry of them it is enough to require
that {\it only one} of them be time-invariant, provided it is the boundary of a $C^1$-smooth domain
compactly contained in $\Om$ (see \cite{MS1}). We also point out that this result can be obtained under 
quite general regularity assumptions on $\pa\Om$, as shown in \cite{MS7}. The case in which $\pa\Om$
is not bounded is not settled: we mention \cite{MS3} and \cite{MS5} for some partial results in this direction and \cite{MS8} for a characterization in $3$-space. 
\par
The possible non-regularity of $\pa\Om$, instead, opens up the way to non-spherical solutions 
if a homogeneous Neumann boundary condition is assumed on $\pa\Om$ for all times. As shown in \cite{Sa},
in fact, truncated spherical cylinders are possible solutions, if $\Om$ is a Lipschitz domain.  At
this date, no results are known if a homogeneous Robin boundary condition is assumed:
$$
u_\nu-\si\,u=0 \ \mbox{ on } \ \pa\Om\times(0,\infty).
$$
\par
We notice that the situations considered in Klamkin's conjecture and in \cite{Sa} rule
out the occurrence of items (ii) and (iii) in Theorem \ref{th:stationary}. Thus one has to deal with (i),
which means that the relevant time-invariant surfaces can only be (portions) of spheres, spherical cylinders or hyperplanes. Of course, the same situation reproduces for any choice of initial and/or boundary conditions which rule out (ii) and (iii). For instance, any time-independent boundary
condition coupled with a non-harmonic initial condition rule out (iii).
\par
Circumstances analogous to those of Theorem \ref{th:stationary} emerge in Theorem\ref{th:non-linear}, where the Laplace operator is replaced by the more general quasi-linear operator $\cQ$ defined in 
\eqref{defQ} with coefficients satisfying \eqref{homogeneous}.  In particular, when case (i) of Theorem \ref{th:non-linear} occurs, in some instances we have to deal with the description of {\it anisotropic} isoparametric functions and surfaces. In the rest of this paragraph, we shall focus on the important case
of the evolution $h$-Laplace equation \eqref{H-Laplace}; here, it is convenient to look at the function $h$ as the {\it support function} of a convex body $K\subset\RN$ belonging to the class $C_+^2$ of bodies having a $C^2$-smooth boundary with (strictly) positive principal curvatures. By setting
\begin{equation}
\label{defH}
H=\frac12\,h^2,
\end{equation}
the system of equations that appears in item (i) of Theorem \ref{th:non-linear} can be conveniently re-written as
\begin{eqnarray}
&&D H(D\fhi)\cdot D\fhi=f(\fhi) \ \mbox{ in } \Om, \label{K-grad} \\ 
&&\De_h \fhi=g(\fhi) \ \mbox{ in } \ \Om. \label{K-lapl}
\end{eqnarray}
We also notice that, since $H$ is $2$-homogeneous, by Euler's identity \eqref{K-grad} can be re-written as
\begin{equation}
\label{K-grad2}
2\,H(D\fhi)=f(\fhi) \ \mbox{ in } \ \Om.
\end{equation}

\par
We shall say that a function $\fhi$ is {\it $K$-isoparametric} if it is a solution of  \eqref{K-grad}-\eqref{K-lapl}; accordingly, its level surfaces will be
called  {\it $K$-isoparametric surfaces}.

%

\begin{thm}[$K$-isoparametric functions]
\label{th:H-isoparametric}
Let $K\subset\RN$ be a convex body of class $C_+^2$ and let $h$ denote its support function. 
\par
Let $\fhi\in C^2(\Om)$ be a $K$-isoparametric function.
Then its level surfaces are of the form
\begin{equation}
\label{K-isoparametric}
Dh(\mathbb S^{M})\times \mathbb R^{N-1-M},\quad M=0,\ldots, N-1.\\
\end{equation}
\end{thm}

\begin{remark}{\rm
Equation \eqref{K-isoparametric} should be interpreted as follows: $Dh(\mathbb S^{M})$ is an $M$-dimensional submanifold of $Dh(\mathbb S^{N-1})=\pa K$ and the vector-valued function $\psi:Dh(\mathbb S^{M})\times \mathbb R^{N-1-M}\to \RN$ defined by $\psi(Dh(\nu),y)=Dh(\nu)+j(y)$, where $j$ is the natural inclusion of $\mathbb R^{N-1-M}$ in $\RN$, defines an embedding and its image coincides, up to homoteties, with a level surface of $\fhi$.
\par
The proof of the Theorem \ref{th:H-isoparametric} relies on the results obtained in \cite{GM} and \cite{HLMG}, which
generalize the classical ones of Levi-Civita and Segre (\cite{Le}, \cite{Se}). In this new setting, the metric defined on the submanifolds (the level sets of the function $\fhi$) is given by an anisotropic (non constant) operator.
In \cite{GM}, the authors prove a classification theorem for hypersurfaces with constant anisotropic principal curvatures; the proof is mainly based on a Cartan-type identity which forces the $K$-isoparametric surface to admit at most two different values for principal curvatures. 
}
\end{remark}
\begin{proof}
Let $\fhi(x)$ be a regular value of the function $\fhi$, $\Si$ the level surface $\{ y\in\Om: \fhi(y)=\fhi(x)\}$ and let $T_x(\Si)$ be the tangent space to $\Si$ at $x$. 
We also introduce the 
 {\it $K$-anisotropic Weingarten operator},
\begin{equation}
\label{defW}
W=D^2h\left(\frac{D\fhi}{|D\fhi|}\right)\frac{D^2\fhi}{|D\fhi|},
\end{equation}
and the {\it $K$-anisotropic mean curvature}, 
\begin{equation}
\label{defM}
M=\frac1{h(D\fhi)}\,\left\{\De_h \fhi-\frac{D\fhi\cdot[D^2H(D\fhi)]\,[D^2\fhi]\,D\fhi}{|D\fhi|^2}\right\}.
\end{equation}
\par
We differentiate \eqref{K-grad} and \eqref{K-grad2} and obtain the identities (by the square brackets, we denote matrices):
\begin{eqnarray}
\label{identities}
&&[D^2H(D\fhi)] [D^2\fhi]\, D\fhi+[D^2\fhi]\, D H(D \fhi) =f'(\fhi)\,D\fhi, \nonumber\\
&&2\,[D^2\fhi]\, DH(D \fhi)=f'(\fhi)\,D\fhi,\\
&&[D^2H(D\fhi)]\,[D^2\fhi]\, D\fhi=[D^2\fhi]\,DH(D \fhi) \nonumber.
\end{eqnarray}
 \par
After straightforward computations, from the definition \eqref{defM}, the identities \eqref{K-lapl}, \eqref{K-grad2} and \eqref{identities} imply that
$$
M=\frac{g(\fhi)-f'(\fhi)/2}{\sqrt{f(\fhi)}},
$$
that means that $M$ is constant on $\Si$.

We are now going to show that $M$ is actually the trace of the $K$-anisotropic Weingarten operator; to do this we prove the following identity
\begin{equation}
\label{shape}
W(x)=\frac{D^2H(D\fhi(x)) D^2\fhi(x)}{h(D\fhi(x))} \ \mbox{ on } \ T_x(\Si);
\end{equation}
in other words, we show that the two matrices coincide as bilinear forms on $T_x(\Si)$.

In fact, \eqref{defH} and the homogeneities of $h$, $DH$ and $D^2H$ imply that
\begin{multline*}
\frac{[D^2h(\nu)] [D^2\fhi]}{|D\fhi|}=
\frac{[D^2H(\nu)] [D^2\fhi]}{h(\nu)\,|\nabla \fhi|}-
\frac{[DH(\nu)\otimes DH(\nu)] [D^2\fhi]}{h(\nu)^3\,|\nabla\fhi|}= \\
\frac{[D^2H(D\fhi)] [D^2\fhi]}{h(D\fhi)}-
\frac{[DH(D\fhi)\otimes DH(D\fhi)][D^2\fhi]}{h(D\fhi)^3}= \\
=\frac{[D^2H(D\fhi)] [D^2\fhi]}{h(D\fhi)}-\frac12\, g'(\fhi)\,
\frac{[DH(D\fhi)\otimes D\fhi]}{h(D\fhi)^3},
\end{multline*}
where, in the last equality, we used the second identity in \eqref{identities}.
The desired formula \eqref{shape} is then obtained by 
noticing that $T_x(\Si)$ lies in the kernel of $[DH(D\fhi(x))\otimes D\fhi(x)]$, being orthogonal to $D\fhi(x)$.
\par
Notice that $M$ only depends on the geometry of the level surface; indeed, 
$
\nu(x)=D\fhi(x)/|D\fhi(x)|
$
is the normal unit vector to $\Si$ at $x$ and the restriction of 
$
-[D^2\fhi(x)]/|D\fhi(x)|
$ 
to $T_x(\Si)$ is the shape operator of $\Si$.

\par
Now, we claim that there exist a relatively compact neighborhood $U_x\subset\Si$ of $x$ and
a number $\de>0$ such that, for any $y\in U_x$, it holds that
\begin{equation}
\label{claim}
\fhi(y+\tau\,DH(D\fhi(y)))=\fhi(x+\tau\,DH(D\fhi(x))),
\end{equation}
for every $0<\tau<\de$.

Without loss of generality we can assume that $f(\fhi)=1$; indeed, since $\Si$ is a regular level surface for $\fhi$, then $f(\fhi)>0$; by taking $\psi=F(\fhi)$ with $F$ such that $(F')^2=f$, it is easy to show that $\psi$ is another isoparametric function, with the same level surfaces of $\fhi$, and such that $H(D\psi)=1$.

To prove our claim, we first have to show that the integral curves of $DH(D\fhi)$ are geodesics.
Let $U_x\subset\Si$ be a relatively compact neighborhood of $x$ (of course, $D\fhi$ does not vanish on $U_x$). For every $y\in U_x$, let $\ga_y(\tau)$ be the solution of the Cauchy problem
$$
\ga_y'(\tau)=DH(D\fhi(\ga_y(\tau)))
, \ \ \gamma_y(0)=y,
$$
and let $\de$ be such that, for every $y\in U_x$, $\ga_y$ remains regular on $[0,\de]$.
Then we have that
\begin{multline*}
\fhi(\ga_y(\tau))-\fhi(y)=\int_0^\de D\fhi(\ga_y(\si))\cdot\ga_y'(\si)\,d\si=\\
\int_0^\de D\fhi(\ga_y(\si))\cdot DH(D\fhi(\ga_y(\si)))\,d\si=2\,\int_0^\de H(\fhi(\ga_y(\si)))\,d\si=
2\de,
\end{multline*}
where we used Euler identity for $H$ and the fact that we are assuming that $f=1$.
 \par
Moreover, we compute that
\begin{multline*}
\ga_y''(\tau)=[D^2H(D\fhi(\ga_y(\tau)))] [D^2\fhi(\ga_y(\tau))]\,\ga_y'(\tau)=\\
[D^2H(D\fhi(\ga_y(\tau)))] [D^2\fhi(\ga_y(\tau))]\,DH(D\fhi(\ga_y(\tau)))=\\
=\frac12\,f'(\fhi(\ga_y(\tau)))\,[D^2H(D\fhi(\ga_y(\tau)))]\,D\fhi(\ga_y(\tau))=0,
\end{multline*}
where we used \eqref{identities} and the fact that we are assuming that $f=1$.
\par
Thus,
$$
\ga_y(\tau)=y+\tau\,DH(D\fhi(y)) \ \mbox{ for } \ 0\le \tau<\de,
$$
and hence 
$$
\fhi(y+\tau\,DH(D\fhi(y)))-\fhi(y)= 2\de \ \mbox{ for } \ 0\le \tau<\de,
$$
which means that \eqref{claim} holds. 

Therefore, we have proved that, for every $\tau\in[0,\de)$, the surfaces 
$$
\Si_\tau=\{y:\fhi(y)=\fhi(\ga_x(\tau)\}
$$ 
are parallel with respect to the anisotropic metric induced by $K$; also, every $\Si_\tau$ has constant $K$-anisotropic mean curvature. 
\par
Thus, \eqref{K-isoparametric} follows the results \cite[Theorem 2.1]{GM} and \cite[Theorem 1.1]{GM}
recalled in the remark.
\end{proof}


\begin{thebibliography}{HLMG}

\bibitem[Al1]{Al1} {\sc G.~Alessandrini},
{\em Matzoh ball soup: a symmetry result for the heat equation},
J. d'Anal. Math. 54 (1990), 229--236.

\bibitem[Al2]{Al2} {\sc G.~Alessandrini}, {\em Characterizing spheres
by functional relations
on solutions of elliptic and parabolic equations}, Appl. Anal. 40
(1991), 251--261.

\bibitem[GM]{GM} {\sc J.~Q.~Ge \& H.~Ma}, {\em Anisotropic isoparametric hypersurfaces in Euclidean spaces}, Ann. Global Anal. Geom. 41 (2012), 347–355.

\bibitem[HLMG]{HLMG} {\sc Y.~J.~He, H.~Z.~Li, H.~Ma \& J.~Q.~Ge}, {\em Compact embedded hypersurfaces with constant
higher order anisotropic mean curvatures}, Indiana Univ. Math. J. 58 (2009), 853–868.

\bibitem[Kl]{Kl} {\sc M.~S.~Klamkin}, {\em A physical characterization of a
sphere}, (Problem 64-5$^*$) SIAM Review 6 (1964), 61; also in Problems
in Applied Mathematics. Selection from SIAM Review. Edited by M.S. Klamkin. 
SIAM, Philadelphia, PA. 1990.

\bibitem[LU]{LU} {\sc O.A. ~Ladyzenskaya, N. N. ~Uraltseva}, {\em Linear and Quasilinear Elliptic
Equations}, Academic Press, New York, 1968. 

\bibitem[Le]{Le} {\sc T.~Levi-Civita}, {\em Famiglie di superficie
isoparametriche nell'ordinario spazio euclideo,}
Atti Accad. Naz. Lincei Rend. Cl. Sci. Fis. Mat. Natur.
26 (1937), 355-362.

\bibitem[MPeS]{MS8} {\sc R.~Magnanini, D. Peralta-Salas \& S.~Sakaguchi}, {\em Stationary isothermic surfaces in Euclidean 3-space},  to appear in Math. Ann., DOI:10.1007/s00208-015-1212-1.

\bibitem[MPrS]{MPS} {\sc R.~Magnanini, J.~Prajapat  \& S.~Sakaguchi}, 
{\em Stationary isothermic surfaces and uniformly dense domains}, Trans. Amer. Math. Soc. 385 (2006), 4821-4841.

\bibitem[MS1]{MS1} {\sc R.~Magnanini \& S.~Sakaguchi}, {\em Matzoh ball soup: heat
conductors with a stationary isothermic surface}, Ann. Math. 156 (2002), 931--946.

\bibitem[MS2]{MS2} {\sc R.~Magnanini \& S.~Sakaguchi}, {\em Interaction
between degenerate diffusion and shape of domain}, Proc. Roy. Soc. Edin. 137A (2007), 373-388.

\bibitem[MS3]{MS3} {\sc R.~Magnanini \& S.~Sakaguchi}, Stationary isothermic surfaces for unbounded domains, Indiana Univ. Math. Journ. 56, no. 6 (2007), 2723-2738.

\bibitem[MS4]{MS4} {\sc R.~Magnanini \& S.~Sakaguchi}, {\em Nonlinear diffusion with a bounded 
stationary  level surface}, Ann. Inst. Henri Poincar\'e Analyse Nonlin\'eaire 27 (2010), 937-952. 

\bibitem[MS5]{MS5} {\sc R.~Magnanini \& S.~Sakaguchi}, {\em Stationary isothermic surfaces
and some characterizations of the hyperplane in $N$-dimensional space}, J. Diff. Eqs. 248 (2010),
1112-1119.

\bibitem[MS6]{MS6} {\sc R.~Magnanini \& S.~Sakaguchi}, {\em Interaction between nonlinear diffusion and geometry of domain}, J. Diff. Eqs. 252 (2012), 236--257.

\bibitem[MS7]{MS7} {\sc R.~Magnanini \& S.~Sakaguchi}, {\em Matzoh ball soup revisited: the boundary regularity issue},  Math. Meth. Appl. Sci. published online, June 18, 2012, DOI: 10.1002/mma.1551.


\bibitem[Sa]{Sa} {\sc S.~Sakaguchi}, {\em When are the spatial level
surfaces of solutions of diffusion equations invariant with respect
to the time variable?}, J. d' Anal. Math. 78 (1999), 219--243.

\bibitem[Se]{Se} {\sc B.~Segre}, {\em Famiglie di ipersuperficie
isoparametriche negli spazi euclidei ad un qualunque numero di
dimensioni},
Atti Accad. Naz. Lincei Rend. Cl. Sci. Fis. Mat. Natur.
27 (1938), 203-207.

\bibitem[Za]{Za} {\sc L.~Zalcman}, {\em Some inverse problems of potential theory},
Cont. Math. 63 (1987), 337--350.



\end{thebibliography}
\end{document}